\documentclass{proc-l}
\usepackage{amsmath}
\usepackage{amssymb}
\usepackage{enumerate}
\usepackage{mathrsfs}
\usepackage{color}

\newtheorem{theorem}{Theorem}[section]
\newtheorem{lemma}[theorem]{Lemma}
\newtheorem{definition}[theorem]{Definition}
\newtheorem{example}[theorem]{Example}

\newtheorem{conjecture}[theorem]{Conjecture}
\newtheorem{proposition}[theorem]{Proposition}
\newtheorem{corollary}[theorem]{Corollary}
\newtheorem{remark}[theorem]{Remark}

\numberwithin{equation}{section}

\def\C{{\mathbb{C}}}
\def\N{{\mathbb{N}}}

\def\Y{{\mathbb{Y}}}

\def\D{{\sigma}}
\def\bu{{\mathbf{u}}}
\def\bv{{\mathbf{v}}}

\def\supp{\hbox{\rm{supp}}}
\def\ld{\hbox{\rm{ld}}}
\def\maximal{\hbox{\rm{max}}}

\begin{document}

\title{Monomial Difference Ideals}
\author{Jie Wang}
\address{KLMM, Academy of Mathematics and Systems Science, The Chinese Academy of Sciences, Beijing 100190, China}
\email{wangjie212@mails.ucas.ac.cn}
\date{January 14, 2016}

\begin{abstract}
In this paper, basic properties of monomial difference ideals are studied. We prove the finitely generated property of well-mixed difference ideals generated by monomials. Furthermore, a finite prime decomposition of radical well-mixed monomial difference ideals is given. As a consequence, we prove that every strictly ascending chain of radical well-mixed monomial difference ideals in a difference polynomial ring is finite, which answers a question raised by E. Hrushovski in the monomial case. Moreover, the Alexander Duality for monomial ideals is generalized to the difference case.
\end{abstract}

\maketitle
\bibliographystyle{amsplain}

\section{Introduction}
Monomial ideals in a polynomial ring have been extendedly studied because of their connections with combinatorics since 1970s. Another reason to study monomial ideals is the fact that they appear as initial ideals of arbitrary ideals. Richard Stanley was the first to use squarefree monomial ideals to study simplicial complexes (\cite{stanley}). Since then, the study of squarefree monomial ideals has become an active research area in combinatorial commutative algebra. In this paper, we study the basic properties of monomial difference (abbr. $\D$-) ideals, and hope that they will play similar role in the study of general $\D$-ideals in a $\D$-polynomial ring.

It is well-known that Hilbert's basis theorem does not hold for $\D$-ideals in a $\D$-polynomial ring. Instead, we have Ritt-Raudenbush basis theorem which asserts that every perfect $\D$-ideal in a $\D$-polynomial ring has a finite basis. It is naturally to ask if the finitely generated property holds for more $\D$-ideals. Let $k$ be a $\D$-field and $R$ a finitely $\D$-generated $k$-$\D$-algebra. In \cite[Section 4.6]{Hrushovski1}, Ehud Hrushovski raised the question whether a radical well-mixed $\D$-ideal in $R$ is finitely generated. The question is also equivalent to whether the ascending chain condition holds for radical well-mixed $\D$-ideals in $R$. For the sake of convenience, let us state it as a conjecture:
\begin{conjecture}\label{intro-conj}
Every strictly ascending chain of radical well-mixed $\D$-ideals in $R$ is finite.
\end{conjecture}

Also in \cite[Section 4.6]{Hrushovski1}, Ehud Hrushovski proved that the answer is yes under some additional assumptions on $R$. In \cite{levin}, Alexander Levin showed that the ascending chain condition does not hold if we drop the radical condition. The counter example given by Levin is a well-mixed $\D$-ideal generated by binomials. In \cite[Section 9]{wibmer1}, Michael Wibmer showed that if $R$ can be equipped with the structure of a $k$-$\D$-Hopf algebra, then Conjecture \ref{intro-conj} is valid.

The main result of this paper is that a well-mixed $\D$-ideal generated by monomials in a $\D$-polynomial ring is finitely generated. Furthermore, we give a finite prime decomposition of radical well-mixed monomial $\D$-ideals. As a consequence, Conjucture \ref{intro-conj} is valid for radical well-mixed monomial $\D$-ideals in a $\D$-polynomial ring.

The paper will be organized as follows. In Section 2, we list some basic facts from difference algebra. In Section 3, several basic properties of monomial $\D$-ideals are proved. In Section 4, we will give a counter example which shows that the well-mixed closure of a monomial $\D$-ideal may not be a monomial $\D$-ideal and prove the finitely generated property of well-mixed $\D$-ideals generated by monomials. In Section 5, we give a finite prime decomposition of radical well-mixed monomial $\D$-ideals. In Section 6, a reflexive prime decomposition of perfect monomial $\D$-ideals will be given. At last, in Section 7, the Alexander Duality for monomial ideals is generalized to the difference case.

\section{Preliminaries}
In this section, we list some basic notions and facts from difference algebra. For more details please refer to \cite{wibmer}. All rings in this paper will be assumed to be commutative and unital.

A {\em difference ring} or {\em $\sigma$-ring} for short $(R,\sigma)$, is a ring $R$ together with a ring endomorphism $\sigma\colon R\rightarrow R$. If $R$ is a field, then we call it a {\em difference field}, or a {\em $\sigma$-field} for short. We usually omit $\sigma$ from the notation, simply refer to $R$ as a $\sigma$-ring or a $\sigma$-field. In this paper, $k$ is always assumed to be a $\D$-field of characteristic $0$.

Following \cite{dd-tdv}, we use the notation of symbolic exponents. Let $x$ be an algebraic indeterminate and $p=\sum_{i=0}^s c_i x^i\in\N[x]$. For $a$ in a $\sigma$-ring $R$, denote $a^p = \prod_{i=0}^s (\sigma^i(a))^{c_i}$. It is easy to check that for $p, q\in\N[x], a^{p+q}=a^{p} a^{q}, a^{pq}= (a^{p})^{q}$.

\begin{definition}
Let $R$ be a $\D$-ring. An ideal $I$ of $R$ is called a {\em $\D$-ideal} if for $a\in R$, $a\in I$ implies $a^x\in I$. Suppose $I$ is a $\D$-ideal of $R$, then $I$ is called
\begin{itemize}
  \item {\em reflexive} if $a^x\in I$ implies $a\in I$ for $a\in R$;
  \item {\em well-mixed} if $ab\in I$ implies $ab^x\in I$ for $a,b\in R$;
  \item {\em perfect} if $a^g\in I$ implies $a\in I$ for $a\in R, g\in \N[x]\backslash\{0\}$;
  \item {\em $\D$-prime} if $I$ is reflexive and a prime ideal as an algebraic ideal.
\end{itemize}
\end{definition}

We give some basic properties about these $\D$-ideals and omit the proofs which can be found in \cite[Chapter 1]{wibmer}.
\begin{lemma}
{\color{white} r}
\begin{enumerate}[(i)]
  \item A $\D$-ideal is perfect if and only if it is reflexive, radical, and well-mixed;
  \item A $\D$-prime $\D$-ideal is perfect;
  \item A prime $\D$-ideal is radical, well-mixed.
\end{enumerate}
\end{lemma}

\begin{lemma}
Let $R$ be a $\D$-ring. A $\D$-ideal $I$ of $R$ is perfect if and only if $a^{x+1}\in I$ implies $a\in I$ for $a\in R$.
\end{lemma}

Let $R$ be a $\D$-ring. If $F\subseteq R$ is a subset of $R$, denote the minimal ideal containing $F$ by $(F)$, the minimal $\D$-ideal containing $F$ by $[F]$ and denote the minimal radical $\D$-ideal, the minimal reflexive $\D$-ideal, the minimal well-mixed $\D$-ideal, the minimal radical well-mixed $\D$-ideal, the minimal perfect $\D$-ideal containing $F$ by $\sqrt{F}, F^*, \langle F\rangle$, $\langle F\rangle_r$, $\{F\}$ respectively, which are called the {\em radical closure}, the {\em reflexive closure}, the {\em well-mixed closure}, the {\em radical well-mixed closure}, the {\em perfect closure} of $F$ respectively.

It can be checked that $\sqrt{F}^*=\sqrt{F^*}$ and $\{F\}=\langle F\rangle_r^*$.

Let $k$ be a $\sigma$-field. Suppose $y=\{y_1,\ldots,y_n\}$ is a set of $\sigma$-indeterminates over $k$. Then the {\em $\sigma$-polynomial ring} over $k$ in $y$ is the polynomial ring in the variables $y,\sigma(y),\sigma^2(y),\ldots$. It is denoted by
\[k\{y\}=k\{y_1,\ldots,y_n\}\]
and has a natural $k$-$\sigma$-algebra structure.

\section{Basic Properties of Monomial Difference Ideals}
In the rest of this paper, unless otherwise specified, $R$ always refers to the $\D$-polynomial ring $k\{y_1,\ldots,y_n\}$. Denote $\N^*=\N\backslash\{0\}$ and $\N[x]^*=\N[x]\backslash\{0\}$.
\begin{definition}
Suppose $\mathbf{u}=(u_1,\ldots,u_n)\in \N[x]^n$. A {\em monomial} in $R$ is a product $\Y^{\mathbf{u}}=y_1^{u_1}\ldots y_n^{u_n}$. A $\D$-ideal $I\subseteq R$ is called a {\em monomial $\D$-ideal} if it is generated by monomials.
\end{definition}

As a vector space over $k$, we can write the $\D$-polynomial ring $R$ as
\[R=k[\N[x]^n]=\oplus_{\mathbf{u}\in \N[x]^n}R_{\mathbf{u}}=\oplus_{\mathbf{u}\in \N[x]^n}k\Y^{\mathbf{u}},\]
where $R_{\mathbf{u}}=k\Y^{\mathbf{u}}$ is the vector subspace of $R$ spanned by the monomial $\Y^{\mathbf{u}}$. Since $R_{\mathbf{u}}\cdot R_{\mathbf{v}}\subseteq R_{\mathbf{u}+\mathbf{v}}$, we see that $R$ is an $\N[x]^n$-graded ring. A monomial $\D$-ideal $I$ defined above is just a graded $\D$-ideal of $R$, which means there exists a subset $S\subseteq \N[x]^n$ such that $I=k[S]:=\oplus_{\mathbf{u}\in S}k\Y^{\mathbf{u}}$. Such $S$ is called the {\em support set} of $I$.

For a set of monomials $F$, we denote $C(F)=\{\mathbf{u}\in \N[x]^n\mid \Y^{\mathbf{u}}\in F\}$.

\begin{lemma}
Let $S\subseteq \N[x]^n$. Then $I=k[S]$ is a monomial $\D$-ideal if and only if $S$ satisfies
\begin{enumerate}[(1)]
\item if $\mathbf{u}\in S, \mathbf{v}\in \N[x]^n$, then $\mathbf{u}+\mathbf{v}\in S$;
\item if $\mathbf{u}\in S$, then $x\mathbf{u}\in S$.
\end{enumerate}
\end{lemma}
\begin{proof}
The necessity is obvious. For the sufficiency, let $a\in I$ and $b\in R$. Then $C(\supp(a))\subset S$. So by (1) and (2), $C(\supp(ab))\subset S$ and $C(\supp(a^x))\subset S$ which implies that $ab\in I$ and $a^x\in I$.
\end{proof}

A subset $S\subseteq \N[x]^n$ satisfying the above conditions is called a {\em character set}.

For $a\in R$, if we write $a=\sum a_{\bu}\Y^{\mathbf{u}}, a_{\bu}\in k$, then
\[\supp(a):=\{\Y^{\mathbf{u}}\mid a_{\bu}\neq 0\}\]
is called the {\em support} of $a$.
\begin{lemma}
Let $I$ be a $\D$-ideal of $R$. Then the followings are equivalent:
\begin{enumerate}[(a)]
  \item $I$ is a monomial $\D$-ideal;
  \item For $a\in R$, $a\in I$ if and only if $\supp(a)\subset I$.
\end{enumerate}
\end{lemma}
\begin{proof}
(a)$\Rightarrow$(b). One direction is obvious. For the other one, let $a\in I$. There exist monomials $h_1,\ldots,h_m\in I$ and $a_1,\ldots,a_m\in R$ such that $a=\sum_{i=1}^ma_ih_i$. It follows that $\supp(a)\subseteq\cup_{i=1}^m\supp(a_ih_i)$. Note that $\supp(a_ih_i)\subset I$ for all $i$, since each monomial in $\supp(a_ih_i)$ is a multiple of $h_i$ and hence belongs to $I$. It follows that $\supp(a)\subset I$ as desired.

(b)$\Rightarrow$(a). Let $F$ be a set of generators of $I$. Since for any $a\in F$, $\supp(a)\subset I$, it follows that $\cup_{a\in F}\supp(a)$ is a set of monomial generators of $I$.
\end{proof}

\begin{lemma}
If $I_1$ and $I_2$ are monomial $\D$-ideals, then $I_1+I_2$ and $I_1\cap I_2$ are also monomial $\D$-ideals.
\end{lemma}
\begin{proof}
Suppose $I_1=k[S_1]$ and $I_2=k[S_2]$. Then $I_1+I_2=k[S_1\cup S_2]$, $I_1\cap I_2=k[S_1\cap S_2]$.
\end{proof}

If $I=k[S]$ is a monomial $\D$-ideal, then the conditions for $I$ to be radical, reflexive, perfect and prime can be described using the support set $S$. To show this, we first define an order on $\N[x]^n$. Let $f=\sum_{i=0}^lf_ix^i, g=\sum_{i=0}^mg_ix^i\in\N[x]$. Let $k$ be a positive integer with $k>\maximal\{l,m\}$, and set $f_i=0$ for $l+1\leqslant i\leqslant k$, $g_i=0$ for $m+1\leqslant i\leqslant k$. Then define $f<g$ if there exists $r\in\N, 0\le r\le k$ such that $f_i=g_i$ for $i=r+1,\ldots,k$ and $f_r<g_r$. Extend the order $<$ to $\N[x]^n$ by comparing $\bu=(u_1,\ldots,u_n)$ and $\bv=(v_1,\ldots,v_n)\in \N[x]^n$ with respect to the lexicographic order. Obviously, this is a total order on $\N[x]^n$ and has the following properties.
\begin{lemma}
The order $<$ defined above satisfies:
\begin{enumerate}[(1)]
\item For $\bu_1,\bu_2,\bv_1,\bv_2\in\N[x]^n$, if $\mathbf{u}_1<\mathbf{v}_1,\mathbf{u}_2\leqslant\mathbf{v}_2$, then $\mathbf{u}_1+\mathbf{u}_2<\mathbf{v}_1+\mathbf{v}_2$;
\item For $\bu,\bv\in\N[x]^n$, if $\mathbf{u}<\mathbf{v}$, then $x\mathbf{u}<x\mathbf{v}$.
\end{enumerate}
\end{lemma}

Let $a\in R$. We define $\deg(a)$ to be the maximal element with respect to $<$ in $C(\supp(a))$.
\begin{proposition}\label{dpmi-prop1}
Let $I=k[S]$ be a monomial $\D$-ideal of $R$. Then:
\begin{enumerate}[(1)]
\item $I$ is radical if and only if for $m\in \N^*$ and $\mathbf{u}\in \N[x]^n$, $m\mathbf{u}\in S$ implies $\mathbf{u}\in S$;
\item $I$ is reflexive if and only if for $\mathbf{u}\in \N[x]^n$, $x\mathbf{u}\in S$ implies $\mathbf{u}\in S$;
\item $I$ is perfect if and only if for $g\in \N[x]^*$ and $\mathbf{u}\in \N[x]^n$, $g\mathbf{u}\in S$ implies $\mathbf{u}\in S$;
\item $I$ is prime if and only if for $\mathbf{u},\mathbf{v}\in \N[x]^n$, $\mathbf{u}+\mathbf{v}\in S$ implies $\mathbf{u}\in S$ or $\mathbf{v}\in S$.
\end{enumerate}
\end{proposition}
\begin{proof}
(1) ``$\Rightarrow$" follows from the definition of radical ideals.

      ``$\Leftarrow$". Suppose that $a=\sum_{i=1}^ka_i\Y^{\mathbf{u}_i}\in R, a_i\neq 0$, $\mathbf{u}_1<\ldots<\mathbf{u}_k$ and $a^m\in I$. We will show that $a\in I$ by induction on the number of terms of $a$. When $k=1$, $a^m=(a_1\Y^{\mathbf{u}_1})^m=a_1^m\Y^{m\mathbf{u}_1}\in I$. Therefore, $\Y^{m\mathbf{u}_1}\in I$ and hence $m\mathbf{u}_1\in S$. So $\mathbf{u}_1\in S$ or equivalently $\Y^{\mathbf{u}_1}\in I$ which implies $a\in I$. For the inductive step, assume now $k\ge2$. Note that $a^m=a_1^m\Y^{m\mathbf{u}_1}+$the other terms. Since $m\mathbf{u}_1$ is minimal in the set of all possible combinations of $\mathbf{u}_{i_1}+\ldots+\mathbf{u}_{i_m}$, the monomial $\Y^{m\mathbf{u}_1}$ cannot be cancelled in the expression of $a^m$ and hence belongs to $\supp(a^m)$. Since $I$ is a monomial $\D$-ideal, $\supp(a^m)\subseteq I$ and hence $\Y^{m\mathbf{u}_1}\in I$ or equivalently $m\mathbf{u}_1\in S$. So $\mathbf{u}_1\in S$ and $\Y^{\mathbf{u}_1}\in I$. Consider $a'=a-a_1\Y^{\mathbf{u}_1}$ with $k-1$ terms. Since $(a')^m=(a-a_1\Y^{\mathbf{u}_1})^m=a^m-\Y^{\mathbf{u}_1}\cdot\ast\in I$, by the induction hypothesis, $a'\in I$. Thus $a=a'+a_1\Y^{\mathbf{u}_1}\in I$.

(2) ``$\Rightarrow$" follows from the definition of reflexive ideals.

      ``$\Leftarrow$". Suppose that $a=\sum_{i=1}^ka_i\Y^{\mathbf{u}_i}\in R, a_i\neq 0$ and $a^x\in I$. Since $a^x=\sum_{i=1}^ka_i^x\Y^{x\mathbf{u}_i}$ and $I$ is a monomial $\D$-ideal, it follows that $\Y^{x\mathbf{u}_i}\in I$ for every $i$. Therefore, $x\mathbf{u}_i\in S$ and hence $\mathbf{u}_i\in S$ for every $i$ which implies $\Y^\mathbf{u}_i\in S$ for every $i$. Thus $a\in I$.

(3) ``$\Rightarrow$" follows from the definition of perfect ideals.

      ``$\Leftarrow$". Suppose that $a=\sum_{i=1}^ka_i\Y^{\mathbf{u}_i}\in R, a_i\neq 0$, $\mathbf{u}_1<\ldots<\mathbf{u}_k$ and $a^{x+1}\in I$. We will prove that $a\in I$ by induction on the number of terms of $a$. When $k=1$, $a^{x+1}=(a_1\Y^{\mathbf{u}_1})^{x+1}=a_1^{x+1}\Y^{(x+1)\mathbf{u}_1}\in I$. Therefore, $\Y^{(x+1)\mathbf{u}_1}\in I$ and hence $(x+1)\mathbf{u}_1\in S$. So $\mathbf{u}_1\in S$ or equivalently $\Y^{\mathbf{u}_1}\in I$ which implies $a\in I$. For the inductive step, assume now $k\ge2$. Note that $a^{x+1}=a_1^{x+1}\Y^{(x+1)\mathbf{u}_1}+$the other terms. Since $\mathbf{u}_1<\ldots<\mathbf{u}_k$ and $x\mathbf{u}_1<\ldots<x\mathbf{u}_k$, $(x+1)\mathbf{u}_1$ is minimal in the set of all possible combinations of $\mathbf{u}_i+x\mathbf{u}_j$. So the monomial $\Y^{(x+1)\mathbf{u}_1}$ cannot be cancelled in the expression of $a^{x+1}$ and hence belongs to $\supp(a^{x+1})$. Since $I$ is a monomial $\D$-ideal, $\supp(a^{x+1})\subseteq I$ and hence $\Y^{(x+1)\mathbf{u}_1}\in I$ or equivalently $(x+1)\mathbf{u}_1\in S$. So $\mathbf{u}_1\in S$ and $\Y^{\mathbf{u}_1}\in I$. Consider $a'=a-a_1\Y^{\mathbf{u}_1}$ with $k-1$ terms. Since $(a')^{x+1}=a^{x+1}-a_1a^x\Y^{\mathbf{u}_1}-a a_1^x\Y^{x\mathbf{u}_1}+a_1^{x+1}\Y^{(x+1)\mathbf{u}_1}\in I$, by the induction hypothesis, $a'\in I$. Thus $a=a'+a_1\Y^{\mathbf{u}_1}\in I$.

(4) ``$\Rightarrow$" follows from the definition of prime ideals.

      ``$\Leftarrow$". Suppose it's not true, so there exists $a$ and $b$ in $R$ such that $ab\in I$ but $a\notin I$ and $b\notin I$. Let $a$ and $b$ are such a pair such that $\deg(a)+\deg(b)$ is minimal. Since $I$ is a monomial $\D$-ideal, $\supp(ab)\subset I$. In particular, the highest degree term is in $I$. The highest degree term is just the product of the leading terms of $a$ and $b$, which we'll call $\ld(a)$ and $\ld(b)$. So $\ld(a)\cdot\ld(b)\in I$, and since they are monomials, we see that either $\ld(a)$ or $\ld(b)$ in $I$. Without loss of generality, assume it's $\ld(a)$. In that case $(a-\ld(a))\cdot b\in I$, but neither $a-\ld(a)$ nor $b$ is in $I$, and this violates the minimality of the pair $a$ and $b$.
\end{proof}

Suppose that $I=k[S]$ is a monomial $\D$-ideal. If $I$ is radical, reflexive, well-mixed, perfect, or prime, then we call the corresponding support set $S$ to be {\em radical}, {\em reflexive}, {\em well-mixed}, {\em perfect}, {\em prime} respectively.

Let $S$ be a subset of $\N[x]^n$. Denote
\begin{align*}
[S]&=\{x^i\mathbf{u}+t\mid \mathbf{u}\in S, i\in\N, t\in\N[x]^n\}\\
&=\{g\mathbf{u}+t\mid \mathbf{u}\in S, g\in\N[x]^*, t\in\N[x]^n\},
\end{align*}
and
\begin{align*}
\sqrt{S}&=\{\mathbf{u}\in \N[x]^n\mid m\mathbf{u}\in [S], m\in \N^*\},\\
S^*&=\{\mathbf{u}\in \N[x]^n\mid x^m\mathbf{u}\in [S], m\in \N\},\\
\{S\}&=\{\mathbf{u}\in \N[x]^n\mid g\mathbf{u}\in [S], g\in \N[x]^*\}.
\end{align*}

One can check that $[\Y^{\mathbf{u}}:\mathbf{u}\in S]=k[\big[S\big]]$ and if $I=k[S]$ is a monomial $\D$-ideal, then $[S]=S$.
\begin{proposition}
Let $I=k[S]$ be a monomial $\D$-ideal of $R$. Then $\sqrt{I}=k[\sqrt{S}]$, $I^*=k[S^*]$, and $\{I\}=k[\{S\}]$.
\end{proposition}
\begin{proof}
Clearly, $k[\sqrt{S}]\subseteq\sqrt{I}$, $k[S^*]\subseteq I^*$, and $k[\{S\}]\subseteq\{I\}$. So we only need to show that $k[\sqrt{S}]$, $k[S^*]$, $k[\{S\}]$ are a radical $\D$-ideal, a reflexive $\D$-ideal, a perfect $\D$-ideal respectively.

Suppose $\mathbf{u}\in \sqrt{S}$ and $\mathbf{v}\in \N[x]^n$, then there exists $m\in \N^*$ such that $m\mathbf{u}\in S$. So $m(\mathbf{u}+\mathbf{v})=m\mathbf{u}+m\mathbf{v}\in S, m(x\mathbf{u})=x(m\mathbf{u})\in S$ and hence $\mathbf{u}+\mathbf{v},x\mathbf{u}\in \sqrt{S}$. Therefore, $k[\sqrt{S}]$ is a $\D$-ideal. Suppose $m\in \N^*$ and $m\mathbf{u}\in \sqrt{S}$, then there exists $m'\in \N^*$ such that $m'm\mathbf{u}\in S$, it follows $\mathbf{u}\in \sqrt{S}$ and thus by Proposition \ref{dpmi-prop1}(1), $k[\sqrt{S}]$ is radical.

Suppose $\mathbf{u}\in S^*$ and $\mathbf{v}\in \N[x]^n$, then there exists $m\in \N$ such that $x^m\mathbf{u}\in S$. So $x^m(\mathbf{u}+\mathbf{v})=x^m\mathbf{u}+x^m\mathbf{v}\in S, x^m(x\mathbf{u})=x^{m+1}\mathbf{u}\in S$ and hence $\mathbf{u}+\mathbf{v},x\mathbf{u}\in S^*$. Therefore, $k[S^*]$ is a $\D$-ideal. Suppose $m\in \N$ and $x^m\mathbf{u}\in S^*$, then there exists $m'\in \N$ such that $x^{m'+m}\mathbf{u}\in S$, it follows $\mathbf{u}\in S^*$ and thus by Proposition \ref{dpmi-prop1}(2), $k[S^*]$ is reflexive.

Suppose $\mathbf{u}\in \{S\}$ and $\mathbf{v}\in \N[x]^n$, then there exists $g\in \N[x]^*$ such that $g\mathbf{u}\in S$. So $g(\mathbf{u}+\mathbf{v})=g\mathbf{u}+g\mathbf{v}\in S, g(x\mathbf{u})=x(g\mathbf{u})\in S$ and hence $\mathbf{u}+\mathbf{v},x\mathbf{u}\in \{S\}$. Therefore, $k[\{S\}]$ is a $\D$-ideal. Suppose $g\in \N[x]^*$ and $g\mathbf{u}\in \{S\}$, then there exists $g'\in \N[x]^*$ such that $g'g\mathbf{u}\in S$, it follows $\mathbf{u}\in \{S\}$ and thus by Proposition \ref{dpmi-prop1}(3), $k[\{S\}]$ is perfect.
\end{proof}

\section{Properties of Well-Mixed $\D$-Ideals Generated by Monomials}
Unlike the radical closure, the reflexive closure, and the perfect closure, the well-mixed closure of a monomial $\D$-ideal may not be a monomial $\D$-ideal. More precisely, it relies on the action of the difference operator. We will give an example to illustrate this. First let us give a concrete description of the well-mixed closure of a $\D$-ideal which has been mentioned in \cite{levin}. Suppose $F$ is a subset of any $\D$-ring $R$. Let $F'=\{a\sigma(b)\mid ab\in F\}$. Note that $F\subset F'$. Let $F^{[0]}=F$ and recursively define $F^{[k]}=(F^{[k-1]})'(k=1,2,\ldots)$. Then by Lemma 3.1 in \cite{levin}, the well-mixed closure of $F$ is
\begin{equation}\label{pmgm-equ3}
\langle F\rangle=\cup_{k=0}^{\infty}F^{[k]}.
\end{equation}
\begin{example}
Let $k=\C$ and $R=\C\{y_1,y_2\}$. Let us consider the $\D$-ideal $I=\langle y_1^2,y_2^2\rangle$ of $R$. Owing to the above process of obtaining the well-mixed closure, we see that $1,y_1,y_1^x,y_2,y_2^x,y_1y_2$ cannot appear in $\supp(c)$ for any $c\in I$. Since $y_1^2-y_2^2=(y_1+y_2)(y_1-y_2)\in I$, we have $(y_1+y_2)(y_1-y_2)^x=y_1^{x+1}+y_1^xy_2-y_1y_2^x-y_2^{x+1}\in I$. Note that $y_1^{x+1},y_2^{x+1}\in I$, hence $y_1^xy_2-y_1y_2^x\in I$. We will show that if the difference operator on $\C$ is the identity map, then $y_1^xy_2,y_1y_2^x\notin I$. As a consequence, $I$ is not a monomial $\D$-ideal.

Set $F=\{y_1^2,y_2^2\}$. By (\ref{pmgm-equ3}), $I=\cup_{k=0}^{\infty}F^{[k]}$. To show $y_1^xy_2,y_1y_2^x\notin I$, we will prove that for any $c\in I$, $y_1y_2^x-y_1^xy_2$ always appears in $c$ as a whole by induction on $k$. $k=0$ is obvious. Assume that for any $c\in F^{[k]}$, $y_1y_2^x-y_1^xy_2$ always appears in $c$ as a whole, and so does $(F^{[k]})$. Now for any $c\in F^{[k+1]}$, suppose $a=a_0+a_1y_1+a_2y_2+a_3y_1^xy_2+a_4y_1^xy_2+*$ and $b=b_0+b_1y_1+b_2y_2+*$, where $a_i,b_j\in \C,0\le i\le 4,0\le j\le2$ and $*$ represents other terms, such that $ab\in (F^{[k]})$, $c=ab^x$ and $y_1y_2^x\textrm{ or }y_1^xy_2\in\supp(c)$. It is impossible that $a_0,b_0\ne 0$ or $a_0\ne0,b_0=0$ since $1,y_1^x,y_2^x$ cannot appear in $\supp(c)$. If $a_0=0,b_0\ne0$, then $a_1,a_2=0$ since $y_1,y_2$ cannot appear in $\supp(c)$ and therefore $ab^x=b_0(a_3y_1^xy_2+a_4y_1y_2^x)+*$. The fact that $ab=b_0(a_3y_1^xy_2+a_4y_1y_2^x)+*\in (F^{[k]})$ implies that $b_0(a_3y_1^xy_2+a_4y_1y_2^x)$ is some multiple of $y_1y_2^x-y_1^xy_2$. Assume now $a_0,b_0=0$. Then $ab=(a_2b_1+a_1b_2)y_1y_2+*$. Since $y_1y_2\notin\supp(ab)$, we have $a_2b_1+a_1b_2=0$. So $c=ab^x=a_2b_1y_1^xy_2+a_1b_2y_1y_2^x+*=a_2b_1(y_1^xy_2-y_1y_2^x)+*$ and hence $y_1y_2^x-y_1^xy_2$ appears in $c$ as a whole.

On the other hand, if the difference operator on $\C$ is the conjugation map(that is $\D(i)=-i$), the situation is totally changed. Since $y_1^2+y_2^2=(y_1+iy_2)(y_1-iy_2)\in I$, $(y_1+iy_2)(y_1-iy_2)^x=y_1^{x+1}+iy_1^xy_2+iy_1y_2^x-y_2^{x+1}\in I$ and hence $y_1^xy_2+y_1y_2^x\in I$. Since we also have $y_1^xy_2-y_1y_2^x\in I$, then $y_1^xy_2,y_1y_2^x\in I$. It follows $I=[y_1^u,y_1^{w_1}y_2^{w_2},y_2^v:2\preceq u,2\preceq v,x+1\preceq w_1+w_2]$($\preceq$ is defined below). In this case, $I=\langle y_1^2,y_2^2\rangle$ is indeed a monomial $\D$-ideal.
\end{example}

In the rest of this section, we will prove that a well-mixed $\D$-ideal generated by monomials could be generated by finitely many monomials as a well-mixed $\D$-ideal. For the proof, we need a new order on $\N[x]^n$ and some lemmas.
\begin{definition}
Let $f=\sum_{i=0}^lf_ix^i, g=\sum_{i=0}^mg_ix^i\in\N[x]$. Let $k$ be a positive integer with $k>\maximal\{l,m\}$, and set $f_i=0$ for $l+1\leqslant i\leqslant k$, $g_i=0$ for $m+1\leqslant i\leqslant k$. Then define $f\preceq g$ if $\sum_{j=i}^kf_j\leqslant\sum_{j=i}^kg_j$ for $i=0,\ldots,k$. Note that $\preceq$ is a partial order on $\N[x]$. Extend $\preceq$ to $\N[x]^n$ by defining $\bu=(u_1,\ldots,u_n)\preceq \bv=(v_1,\ldots,v_n)$ if and only if $u_i\preceq v_i$ for $i=1,\ldots,n$.
\end{definition}
It is straightforward from the definition that the partial order $\preceq$ has the following properties.
\begin{lemma}
Let $\mathbf{u}_1,\bu_2,\mathbf{v}_1,\bv_2\in\N[x]^n$. If $\mathbf{u}_1\preceq \mathbf{v}_1, \mathbf{u}_2\preceq \mathbf{v}_2$ ,then $x\mathbf{u}_1\preceq x\mathbf{v}_1, \mathbf{u}_1+\mathbf{u}_2\preceq \mathbf{v}_1+\mathbf{v}_2$.
\end{lemma}

Moreover, we have:
\begin{lemma}\label{drmi-lemma1}
Let $\mathbf{u},\mathbf{v}\in\N[x]^n$. If $\mathbf{u}\preceq \mathbf{v}$, then $\Y^\mathbf{v}\in\langle \Y^{\mathbf{u}}\rangle$.
\end{lemma}
\begin{proof}
For brevity, we just prove the case $n=1$, since the case $n\ge2$ is similar. Let $f,g\in\N[x]$. Without loss of generality, we can assume that $f=\sum_{i=0}^lf_ix^i, g=\sum_{i=0}^lg_ix^i$. We shall show that if $f\preceq g$, then $y_1^g\in\langle y_1^f\rangle$. Let us do this by induction on $l$. When $l=0$, it is clear that $f_0\preceq g_0$ implies that $y_1^{g_0}\in\langle y_1^{f_0}\rangle$. For the inductive step, assume now $l\ge1$. Since $\sum_{i=0}^lf_i\le\sum_{i=0}^lg_i$, increase $f_0$ if necessary such that $\sum_{i=0}^lf_i=\sum_{i=0}^lg_i$. Then we have $f_0\ge g_0$, since $\sum_{i=1}^lf_i\le\sum_{i=1}^lg_i$. Let $f'=\sum_{i=2}^lf_ix^{i-1}+f_1+f_0-g_0,g'=\sum_{i=2}^lg_ix^{i-1}+g_1$. Then $f'\preceq g'$. So by the induction hypothesis, $y_1^{g'}\in\langle y_1^{f'}\rangle$. It follows that $y_1^g=y_1^{g'x+g_0}\in\langle y_1^{f'x+g_0}\rangle\subseteq\langle y_1^{f'x+g_0-(f_0-g_0)x+f_0-g_0}\rangle=\langle y_1^f\rangle$.
\end{proof}

For $f=\sum_{i=0}^lf_ix^i\in \N[x]$, denote $|f|=\sum_{i=0}^l f_i\in\N$.
\begin{lemma}\label{pmgm-lemma1}
Let $S\subseteq \N[x]$ such that $|f|$ is a constant $a\in\N$ for all $f\in S$. Then the set of minimal elements of $S$ with respect to the partial order $\preceq$ is finite.
\end{lemma}
\begin{proof}
We shall prove the lemma by induction on $a$. The case $a=0,1$ is clear. For the inductive step, assume now $a\ge2$. Choose an $f=\sum_{i=0}^lf_ix^i\in S$. The set $G=\{g=\sum_{i=0}^lg_ix^i\mid |g|\leqslant a\}$ is finite. For each $g=\sum_{i=0}^lg_ix^i\in G$, we define
\[S_g=\{\sum_{i=0}^mh_ix^i\in S\mid m>l, \sum_{i=0}^lh_ix^i=g\}\]
and
\[S_g'=\{\sum_{i=l+1}^mh_ix^i\mid\sum_{i=0}^mh_ix^i\in S_g\}.\]
It follows $S=\cup_{g\in G}S_g$. For each $g\in G$, if $g\neq 0$, then for all $h\in S_g'$, $|h|$ is a constant which is lower than $a$. So by the induction hypothesis, the set of minimal elements of $S_g'$ with respect to the partial order $\preceq$ is finite and hence so is $S_g$. Note that for any $h\in S_0$, $f\preceq h$. Therefore, the set of minimal elements of $S$ is contained in the union of the set of minimal elements of $S_g$, where $g\in G,g\ne0$, and $\{f\}$. Since the latter is a finite set, the former must be finite.
\end{proof}

\begin{lemma}\label{pmgm-lemma2}
Let $S\subseteq \N[x]$ such that for all $f\in S, |f|\leqslant a$ for some $a\in\N$. Then the set of minimal elements of $S$ with respect to the partial order $\preceq$ is finite.
\end{lemma}
\begin{proof}
It is an immediate corollary of the above lemma.
\end{proof}

\begin{lemma}\label{pmgm-lemma3}
Let $S\subseteq \N[x]$ such that for all $f\in S, \deg(f)\leqslant k$ for some $k\in\N$. Then the set of minimal elements of $S$ with respect to the partial order $\preceq$ is finite.
\end{lemma}
\begin{proof}
We shall prove the lemma by induction on $k$. The case $k=0$ is clear. For the inductive step, assume now $k\ge1$. Choose an $f=\sum_{i=0}^kf_ix^i\in S$ and denote $c=|f|$. For every $j$, $0\le j<c,j\in\N$, there exists an integer, say $s$, $0\leqslant s\leqslant k$, such that $\sum_{i=s+1}^kf_i\leqslant j<\sum_{i=s}^kf_i$, and we define
\[U_j=\{\sum_{i=s}^kg_ix^i\mid \sum_{i=s}^kg_i=j\}\]
which is obviously a finite set. For each $g\in U_j$, define
\[S_g=\{\sum_{i=0}^kg_ix^i\in S\mid \sum_{i=s}^kg_ix^i=g\}\]
and
\[S_g'=\{\sum_{i=0}^{s-1}g_ix^i\mid \sum_{i=0}^kg_ix^i\in S_g\}.\]
In addition, we define
\[S_f=\{h\in S\mid f\preceq h\}.\]
By the definition of $\preceq$, for an $h\in S$, if $h\npreceq f$, then $h$ must belong to some $S_g$. So we have
\begin{equation}\label{pmgm-equ1}
S=(\cup_{j=0}^{c-1}\cup_{g\in U_j}S_g)\cup S_f.
\end{equation}
Note that for each $g\in U_j$, $\deg(h)<k$ for all $h\in S_g'$, so by the induction hypothesis, the set of minimal elements of $S_g'$ with respect to the partial order $\preceq$ is finite and hence so is $S_g$. Because of (\ref{pmgm-equ1}), we have the set of minimal elements of $S$ is contained in the union of the set of minimal elements of $S_g$ and $\{f\}$, where $g\in U_j$ and $j=0,\ldots,c-1$. Since the latter is a finite set, the former must be finite.
\end{proof}

\begin{lemma}\label{pmgm-lemma4}
Let $S\subseteq \N[x]$. Then the set of minimal elements of $S$ with respect to the partial order $\preceq$ is finite.
\end{lemma}
\begin{proof}
Let us choose an $f=\sum_{i=0}^kf_ix^i\in S$ and denote $c=|f|$. For every $j$, $0\le j<c-f_0,j\in\N$, there exists an integer, say $s$, $1\leqslant s\leqslant k$, such that $\sum_{i=s+1}^kf_i\leqslant j<\sum_{i=s}^kf_i$, and we define
\begin{align*}
U_j&=\{\sum_{i=s}^lg_ix^i\mid \sum_{i=0}^lg_ix^i\in S, \sum_{i=s}^lg_i=j\},\\
S_j&=\{\sum_{i=0}^lg_ix^i\in S\mid \sum_{i=s}^lg_i=j\}
\end{align*}
and
\[S_j'=\{\sum_{i=0}^{s-1}g_ix^i\mid \sum_{i=0}^lg_ix^i\in S_j\}.\]

By Lemma \ref{pmgm-lemma1}, for each $j$, the set of minimal elements of $U_j$ is finite, which is denoted by $V_j$. By Lemma \ref{pmgm-lemma3}, for each $j$, the set of minimal elements of $S_j'$ is finite, which is denoted by $W_j$.

In addition, we define
\begin{align*}
S_c&=\{\sum_{i=0}^lg_ix^i\in S\mid \sum_{i=0}^lg_i<c\},\\
S_f&=\{h\in S\mid f\preceq h\}.
\end{align*}

By the definition of $\preceq$, for an $h\in S$, if $h\notin S_c$ and $h\notin S_f$, then $h$ must belong to some $S_j$. So we have
\begin{equation}\label{pmgm-equ2}
S=(\cup_{j=0}^{c-f_0-1}S_j)\cup S_c\cup S_f.
\end{equation}
By Lemma \ref{pmgm-lemma2}, the set of minimal elements of $S_c$ is finite, which is denoted by $C$. We claim that the set of minimal elements of $S$ is contained in $(\cup_{j=0}^{c-f_0-1}(V_j+W_j))\cup C\cup \{f\}$, where $V_j+W_j$ means $\{g+h\mid g\in V_j, h\in W_j\}$. To prove this, let $g=\sum_{i=0}^lg_ix^i\in S$. By (\ref{pmgm-equ2}), if $g\notin S_c$ and $g\notin S_f$, then there exists $j$ such that $g\in S_j$. By definition, $\sum_{i=s}^lg_ix^i\in U_j$ and $\sum_{i=0}^{s-1}g_ix^i\in S_j'$. So there exists $h\in V_j$ and there exists $h'\in W_j$ such that $h\preceq \sum_{i=s}^lg_ix^i$ and $h'\preceq \sum_{i=0}^{s-1}g_ix^i$. Therefore, $h+h'\in V_j+W_j$ such that $h+h'\preceq\sum_{i=s}^lg_ix^i+\sum_{i=0}^{s-1}g_ix^i=g$, which proves the claim.

Since $(\cup_{j=0}^{c-f_0-1}(V_j+W_j))\cup C\cup \{f\}$ is a finite set, it follows that the set of minimal elements of $S$ is finite.
\end{proof}

\begin{lemma}\label{pmgm-lemma5}
Let $S\subseteq \N[x]^n$. Then the set of minimal elements of $S$ with respect to the partial order $\preceq$ is finite.
\end{lemma}
\begin{proof}
We shall prove the lemma by induction on $n$. The case $n=1$ is proved by Lemma \ref{pmgm-lemma4}. For the inductive step, assume now $n\ge2$. Define
\[S_1=\{u_1\mid(u_1,\ldots,u_n)\in S\}.\]
By Lemma \ref{pmgm-lemma4}, the set of minimal elements of $S_1$ is finite, which is denoted by $U$. For each $u\in U$, define
\[S_u=\{(u_1,u_2,\ldots,u_n)\in S\mid u\preceq u_1\}\]
and
\[S_u'=\{(u_2,\ldots,u_n)\mid (u_1,u_2,\ldots,u_n)\in S_u\}.\]
Obviously, we have $S=\cup_{u\in U}S_u$.

By the induction hypothesis, for each $u\in U$, the set of minimal elements of $S_u'$ with respect to the partial order $\preceq$ is finite, which is denoted by $V_u$. Let $u\times V_u=\{(u,\bv)\mid \bv\in V_u\}$. We claim that the set of minimal elements of $S$ is contained in $\cup_{u\in U}u\times V_u$. To prove this, let $\bu=(u_1,u_2,\ldots,u_n)\in S$. There exists $u\in U$ such that $\bu \in S_u$. By definition, $(u_2,\ldots,u_n)\in S_u'$. So there exists $\bv\in V_u$ such that $\bv\preceq (u_2,\ldots,u_n)$. Therefore, $(u,\bv)\preceq (u_1,u_2,\ldots,u_n)=\bu$ and $(u,\bv)\in u\times V_u$ which proves the claim.

Since $\cup_{u\in U}u\times V_u$ is a finite set, the set of minimal elements of $S$ is finite.
\end{proof}

Now we can prove the finitely generated property of well-mixed $\D$-ideals generated by monomials.
\begin{theorem}\label{pmgm-thm}
Let $I=\langle \Y^{\mathbf{u}}:\bu\in S\rangle$ for some $S\subseteq \N[x]^n$. Then $I$ is generated by a finite set of monomials as a well-mixed $\D$-ideal.
\end{theorem}
\begin{proof}
If $\bu\preceq \bv$, we can delete $\bv$ from the generating set $S$ to get the same well-mixed $\D$-ideal, since $\Y^\mathbf{v}\in\langle \Y^{\mathbf{u}}\rangle$ by Lemma \ref{drmi-lemma1}. So we only need to show that the set of minimal elements of $S$ with respect to the partial order $\preceq$ is finite, which follows from Lemma \ref{pmgm-lemma5}.
\end{proof}

\begin{corollary}\label{pmgm-cor}
Any strictly ascending chain of well-mixed $\D$-ideals generated by monomials in $R$ is finite.
\end{corollary}
\begin{proof}
Assume that $I_1\subseteq I_2\subseteq\ldots\subseteq I_k\ldots$ is an ascending chain of well-mixed $\D$-ideals generated by monomials. Then $\cup_{i=1}^{\infty}I_i$ is also a well-mixed $\D$-ideal generated by monomials. By Theorem \ref{pmgm-thm}, $\cup_{i=1}^{\infty}I_i$ is finitely generated by monomials, say $\{a_1,\dots,a_m\}$. Then there exists $k\in \N$ large enough such that $\{a_1,\dots,a_m\}\subset I_k$. It follows $I_k=I_{k+1}=\ldots=\cup_{i=1}^{\infty}I_i$.
\end{proof}

\begin{remark}
It should be pointed out that a counter example due to Levin in \cite{levin} shows that Corollary \ref{pmgm-cor} does not hold even for well-mixed $\D$-ideals generated by binomials in $R$.
\end{remark}

\section{Prime Decomposition of Radical Well-Mixed Monomial $\D$-Ideals}
In this section, we will give a finite prime decomposition of radical well-mixed monomial $\D$-ideals. First let us prove some lemmas. Notations follow as Section 4.
\begin{lemma}\label{drmi-lemma}
Let $F$ and $G$ be subsets of any $\D$-ring $R$. Then
\begin{enumerate}[(a)]
  \item $F^{[1]}G^{[1]}\subseteq(FG)^{[1]}$;
  \item $F^{[i]}G^{[i]}\subseteq(FG)^{[i]}$ for $i=1,2,\ldots$;
  \item $F^{[i]}\cup G^{[i]}\subseteq\sqrt{(FG)^{[i]}}$ for $i=1,2,\ldots$.
\end{enumerate}
\end{lemma}
\begin{proof}
\begin{enumerate}[(a)]
  \item Let $a\sigma(b)\in F^{[1]}$ and $c\sigma(d)\in G^{[1]}$ such that $ab\in (F)$ and $cd\in (G)$. Then $abcd\in (FG)$ and it follows that $ac\sigma(bd)=a\sigma(b)c\sigma(d)\in (FG)^{[1]}$. So $F^{[1]}G^{[1]}\subseteq(FG)^{[1]}$.
  \item We prove (b) by induction on $i$. The case $i=1$ is proved by (a). For the inductive step, assume now $i\ge2$. Then by (a) and the induction hypothesis,
  \begin{align*}
  F^{[i]}G^{[i]}&=(F^{[i-1]})^{[1]}(G^{[i-1]})^{[1]}\subseteq (F^{[i-1]}G^{[i-1]})^{[1]}\\
  &\subseteq ((FG)^{[i-1]})^{[1]}=(FG)^{[i]}.
  \end{align*}
  \item For any $a\in F^{[i]}\cup G^{[i]}$, $a^2\in F^{[i]}G^{[i]}$. It follows that $a\in \sqrt{F^{[i]}G^{[i]}}\subseteq \sqrt{(FG)^{[i]}}$.
\end{enumerate}
\end{proof}

\begin{proposition}\label{drmi-prop}
Let $F$ and $G$ be subsets of any $\D$-ring $R$. Then
\[\langle F\rangle_r\cap\langle G\rangle_r=\langle FG\rangle_r.\]
As a corollary, if $I$ and $J$ are two $\D$-ideals of $R$, then
\[\langle I\rangle_r\cap\langle J\rangle_r=\langle I\cap J\rangle_r=\langle IJ\rangle_r.\]
\end{proposition}
\begin{proof}
$\langle F\rangle_r\cap\langle G\rangle_r\supseteq \langle FG\rangle_r$ is clear. It is enough to show the converse. Because of (\ref{pmgm-equ3}),
\begin{align*}
  \langle F\rangle_r\cap\langle G\rangle_r&=\sqrt{\langle F\rangle}\cap\sqrt{\langle G\rangle}=\sqrt{\cup_{i=0}^{\infty}F^{[i]})}\cap\sqrt{\cup_{i=0}^{\infty}G^{[i]})}\\
  &=\sqrt{\cup_{i=0}^{\infty}(F^{[i]}\cap G^{[i]})}\subseteq\sqrt{\cup_{i=0}^{\infty}\sqrt{(FG)^{[i]}}}=\langle FG\rangle_r,
\end{align*}
where the inclusion follows from Lemma \ref{drmi-lemma}(c).
\end{proof}

\begin{lemma}\label{ddi-lemma}
Let $I$ be a radical well-mixed $\D$-ideal of $R$. Suppose $\Y^{\mathbf{u}_1},\Y^{\mathbf{u}_2}$ are two monomials in $R$ such that $\Y^{\mathbf{u}_1+\mathbf{u}_2}\in I$. Then
\[I=\langle I, \Y^{\mathbf{u}_1}\rangle_r\cap\langle I, \Y^{\mathbf{u}_2}\rangle_r.\]
\end{lemma}
\begin{proof}
By Proposition \ref{drmi-prop},
\[\langle I, \Y^{\mathbf{u}_1}\rangle_r\cap\langle I, \Y^{\mathbf{u}_2}\rangle_r=\langle I, \Y^{\mathbf{u}_1+\mathbf{u}_2}\rangle_r=I.\]
\end{proof}

For $\mathbf{b}=(b_1,\ldots,b_n)\in (\N\cup\{-1\})^n$, we define
\[\mathfrak{m}^{\mathbf{b}}:=[y_i^{x^{b_i}}\mid b_i\neq -1],\]
which is a prime monomial $\D$-ideal.

For $m\in \N^*$, denote $[m]=\{1,\ldots,m\}$.

Now we can prove the decomposition theorem of radical well-mixed monomial $\D$-ideals.
\begin{theorem}\label{drdi-thm}
Let $I=\langle \Y^{\bu}:\bu\in S\rangle_r$ where $S\subseteq \N[x]^n$. Then $I$ can be written as a finite intersection of prime monomial $\D$-ideals of the forms $\mathfrak{m}^{\mathbf{b}}$. That is, there exist $\mathbf{b}_1,\ldots,\mathbf{b}_s\in (\N\cup\{-1\})^n$ such that
\[I=\mathfrak{m}^{\mathbf{b}_1}\cap\ldots\cap\mathfrak{m}^{\mathbf{b}_s}.\]
Moreover, if the decomposition is irredundant, then it is unique.
\end{theorem}
\begin{proof}
By Lemma \ref{ddi-lemma}, if a monomial $\Y^{\mathbf{u}}\in I$ and $\mathbf{u}=\mathbf{u}_1+\mathbf{u}_2$, then $I=\langle I, \Y^{\mathbf{u}_1}\rangle_r\cap\langle I, \Y^{\mathbf{u}_2}\rangle_r$. Iterating this process eventually write $I$ as follows:
\[I=\cap\langle y_{i_1}^{x^{b_{i_1}}},\ldots, y_{i_k}^{x^{b_{i_k}}}\rangle_r.\]
Note that $[y_{i_1}^{x^{b_{i_1}}}, \ldots, y_{i_k}^{x^{b_{i_k}}}]$ is a prime $\D$-ideal, therefore
\[\langle y_{i_1}^{x^{b_{i_1}}}, \ldots, y_{i_k}^{x^{b_{i_k}}}\rangle_r=[y_{i_1}^{x^{b_{i_1}}}, \ldots, y_{i_k}^{x^{b_{i_k}}}]\]
and
\[I=\cap[y_{i_1}^{x^{b_{i_1}}}, \ldots, y_{i_k}^{x^{b_{i_k}}}].\]

After deleting unnecessary intersectands, we can assume that the intersection is irredundant. Using an argument similar to the proof of Dickson's lemma, we see that this irredundant intersection must be finite. So there exist $\mathbf{b}_1,\ldots,\mathbf{b}_s\in (\N\cup\{-1\})^n$ such that
\[I=\mathfrak{m}^{\mathbf{b}_1}\cap\ldots\cap\mathfrak{m}^{\mathbf{b}_s}.\]

Let $\mathfrak{m}^{\mathbf{b}_1}\cap\ldots\cap\mathfrak{m}^{\mathbf{b}_s}=\mathfrak{m}^{\mathbf{a}_1}\cap\ldots\cap
\mathfrak{m}^{\mathbf{a}_t}$ be two irredundant decompositions of $I$. We will show that for each $i\in [s]$, there exists $j\in [t]$ such that $\mathfrak{m}^{\mathbf{a}_j}\subseteq \mathfrak{m}^{\mathbf{b}_i}$. By symmetry, we then also have that for each $k\in [t]$, there exists $l\in [s]$ such that $\mathfrak{m}^{\mathbf{b}_l}\subseteq \mathfrak{m}^{\mathbf{a}_k}$. This implies that $s=t$ and $\{\mathfrak{m}^{\mathbf{b}_1},\ldots,\mathfrak{m}^{\mathbf{b}_s}\}=\{\mathfrak{m}^{\mathbf{a}_1},\ldots,
\mathfrak{m}^{\mathbf{a}_t}\}$.

In fact, let $i\in [s]$. We may assume that $\mathfrak{m}^{\mathbf{b}_i}=[y_1^{x^{b_{i1}}},\ldots,y_r^{x^{b_{ir}}}]$. Suppose that $\mathfrak{m}^{\mathbf{a}_j}\nsubseteq \mathfrak{m}^{\mathbf{b}_i}$ for all $j\in [t]$. Then for each $j$ there exists $y_{l_j}^{x^{c_j}}\in \mathfrak{m}^{\mathbf{a}_j}\backslash \mathfrak{m}^{\mathbf{b}_i}$. It follows that either $l_j\notin[r]$ or $c_j<b_{il_j}$. Let
\[a=\prod_{j=1}^ty_{l_j}^{x^{c_j}}.\]
We have $a\in \cap_{j=1}^t\mathfrak{m}^{\mathbf{a}_j}\subseteq \mathfrak{m}^{\mathbf{b}_i}$. Therefore, there exists $j\in[t]$ such that $l_j\in[r]$ and $b_{il_j}\le c_j$. This is impossible.
\end{proof}

If $I$ is a radical well-mixed monomial $\D$-ideal, then the irredundant prime decomposition of $I$ obtained in Theorem \ref{drdi-thm} is called the {\em standard prime decomposition} of $I$ and each $\mathfrak{m}^{\mathbf{b}_i}$ is called an {\em irreducible component} of $I$.

\begin{corollary}
The radical well-mixed closure of a monomial $\D$-ideal is still a monomial $\D$-ideal.
\end{corollary}
\begin{proof}
Suppose $I$ is a monomial $\D$-ideal. By Theorem \ref{drdi-thm}, there exist $\mathbf{b}_1,\ldots,\mathbf{b}_s\in (\N\cup\{-1\})^n$ such that $\langle I\rangle_r=\cap_{i=1}^s\mathfrak{m}^{\mathbf{b}_i}$. Since every $\mathfrak{m}^{\mathbf{b}_i}$ is a monomial $\D$-ideal and the intersection of monomial $\D$-ideals is a monomial $\D$-ideal, it follows that $\langle I\rangle_r$ is a monomial $\D$-ideal.
\end{proof}

\begin{corollary}\label{drmi-cor}
Every radical well-mixed monomial $\D$-ideal in $R$ is generated by finitely many monomials as a radical well-mixed $\D$-ideal.
\end{corollary}
\begin{proof}
Suppose $I$ is a radical well-mixed monomial $\D$-ideal. Let $I=\cap_{i=1}^s\mathfrak{m}^{\mathbf{b}_i}$ be the standard prime decomposition of $I$. By Proposition \ref{drmi-prop}, $\cap_{i=1}^s\mathfrak{m}^{\mathbf{b}_i}$ equals to a radical well-mixed $\D$-ideal which is generated by finitely many monomials, so $I$ is finitely generated as a radical well-mixed $\D$-ideal.
\end{proof}

By Corollary \ref{drmi-cor}, for a radical well-mixed monomial $\D$-ideal $I$ of $R$, there exist $\mathbf{a}_1,\ldots,\mathbf{a}_m\in(\N\cap\{-1\})^n$ with $\mathbf{a}_j=(a_{ji})_{i=1}^n, j=1,\ldots,m$ such that
\[I=\langle \prod_{i=1}^n y_i^{x^{a_{1i}}},\ldots,\prod_{i=1}^n y_i^{x^{a_{mi}}}\rangle_r,\]
where we set $x^{-1}=0$. We call $\{\mathbf{a}_1,\ldots,\mathbf{a}_m\}$ the {\em character vectors} of $I$ and call $\{\prod_{i=1}^n y_i^{x^{\mathbf{a}_{1i}}},\ldots,\prod_{i=1}^n y_i^{x^{\mathbf{a}_{mi}}}\}$ the {\em set of minimal generators} of $I$, denoted by $G(I)$.

\begin{corollary}\label{drdi-cor}
Any strictly ascending chain of radical well-mixed monomial $\D$-ideals in $R$ is finite.
\end{corollary}
\begin{proof}
Assume that $I_1\subseteq I_2\subseteq\ldots\subseteq I_k\ldots$ is an ascending chain of radical well-mixed monomial $\D$-ideals. Then $\cup_{i=1}^{\infty}I_i$ is also a radical well-mixed monomial $\D$-ideal. By Corollary \ref{drmi-cor}, $\cup_{i=1}^{\infty}I_i$ is finitely generated by monomials, say $\{a_1,\dots,a_m\}$. Then there exists $k\in \N$ large enough such that $\{a_1,\dots,a_m\}\subset I_k$. It follows $I_k=I_{k+1}=\ldots=\cup_{i=1}^{\infty}I_i$.
\end{proof}

\begin{remark}
By Corollary \ref{drdi-cor}, Conjecture \ref{intro-conj} is valid for radical well-mixed monomial $\D$-ideals.
\end{remark}

In the following, we give a criterion to check if a monomial $\D$-ideal is radical well-mixed using its support set.
\begin{lemma}
An intersection of prime $\D$-ideals is radical well-mixed.
\end{lemma}
\begin{proof}
A prime $\D$-ideal is radical well-mixed and an intersection of radical well-mixed $\D$-ideals is radical well-mixed.
\end{proof}

\begin{corollary}\label{drmi-cor2}
Let $I=k[S]$ be a monomial $\D$-ideal of $R$. Then $I$ is radical well-mixed if and only if the following conditions are satisfied:
\begin{enumerate}[(a)]
  \item For $m\in \N^*$ and $\mathbf{u}\in\N[x]^n$, $m\mathbf{u}\in S$ implies $\mathbf{u}\in S$;
  \item For $\mathbf{u},\mathbf{v}\in\N[x]^n$, $\mathbf{u}+\mathbf{v}\in S$ implies $\mathbf{u}+x\mathbf{v}\in S$.
\end{enumerate}
\end{corollary}
\begin{proof}
``$\Rightarrow$" is clear.

``$\Leftarrow$". For $\mathbf{u}=(u_1,\ldots,u_n)\in \N[x]^n$, define $\deg(\mathbf{u})=(\deg(u_1),\ldots,\deg(u_n))$ and set $\deg(0)=-1$. If $\mathbf{b}=(b_1,\ldots,b_n)\in (\N\cup\{-1\})^n$, then let $x^{\mathbf{b}}=(x^{b_1},\ldots,x^{b_n})$ and set $x^{-1}=0$. So from (a) and (b), we obtain
$$\forall \mathbf{u}\in S\Rightarrow x^{\deg(\mathbf{u})}\in S.$$

Let $U$ be the subset of $S$ which is the set of minimal elements in $\{x^{\deg(\mathbf{u})}\mid \mathbf{u}\in S\}$ with respect to the order $\leqslant$($\bu=(u_i)_{i=1}^n\leqslant\bv=(v_i)_{i=1}^n$ if and only if $\deg(u_i)\leqslant\deg(v_i)$ for all $i$). Using an argument similar to the proof of Dickson's lemma, we see that $U$ is a finite set. Moreover,
\begin{equation}
S=\{\mathbf{v}\in \N[x]^n\mid \mathbf{u}\leqslant\mathbf{v} \textrm{ for some } \mathbf{u}\in U\}.
\end{equation}
Or equivalently,
\begin{equation}\label{drmi-equ1}
S=\cup_{\mathbf{u}\in U}\cap_{i=1}^n\{\mathbf{v}\in \N[x]^n\mid \deg(u_i)\leqslant \deg(v_i)\}.
\end{equation}
Exchange $\cup$ and $\cap$ in (\ref{drmi-equ1}), and we obtain $\mathbf{b}_1,\ldots,\mathbf{b}_s\in (\N\cup\{-1\})^n$ with $\mathbf{b}_i=\{b_{ij}\}_{j=1}^n, i=1,\ldots,s$ such that
\begin{equation}\label{drmi-equ2}
S=\cap_{i=1}^s\cup_{j=1}^n\{\mathbf{v}\in \N[x]^n\mid b_{ij}\leqslant \deg(v_i)\}.
\end{equation}
Because of (\ref{drmi-equ2}), we have
\[I=k[S]=\cap_{i=1}^s\mathfrak{m}^{\mathbf{b}_i}.\]
Since all $\mathfrak{m}^{\mathbf{b}_i}$ are prime $\D$-ideals, $I$ is radical well-mixed.
\end{proof}

Suppose $S$ is a subset of $\N[x]^n$. Let
\[S'=\{\mathbf{u}+x\mathbf{v}\mid \mathbf{u}+\mathbf{v}\in S, \mathbf{u},\mathbf{v}\in \N[x]^n\}.\]
Let $S^{[0]}=S$ and recursively we define $S^{[k]}=[S^{[k-1]}]'(k=1,2,\ldots)$. Denote
\[\langle S\rangle=\cup_{k=0}^{\infty}S^{[k]}.\]
\begin{corollary}
Let $I=k[S]$ be a monomial $\D$-ideal of $R$. Then $\langle I\rangle_r=k[\sqrt{\langle S\rangle}]$.
\end{corollary}
\begin{proof}
Clearly, $\langle I\rangle_r\supseteq k[\sqrt{\langle S\rangle}]$. We just need to show that $k[\sqrt{\langle S\rangle}]$ is already a radical well-mixed $\D$-ideal. It is easy to check that $k[\sqrt{\langle S\rangle}]$ is a $\D$-ideal. To show it is radical well-mixed, we need to check that $\sqrt{\langle S\rangle}$ satisfies the conditions (a) and (b) of Corollary \ref{drmi-cor2}. (a) is obvious. For (b), let $\bu,\bv\in\N[x]^n$ such that $\bu+\bv\in \sqrt{\langle S\rangle}$, then there exists $m\in\N^*$ such that $m(\bu+\bv)\in\langle S\rangle=\cup_{k=0}^{\infty}S^{[k]}$. So there exists $k\in\N$ such that $m(\bu+\bv)\in S^{[k]}$. Hence $m(\bu+x\bv)\in S^{[k+1]}\subseteq\langle S\rangle$. Therefore, $\bu+x\bv\in\sqrt{\langle S\rangle}$.
\end{proof}

\begin{corollary}\label{drmi-cor1}
Suppose $S_1,S_2\subseteq\N[x]^n$. Then $k[\sqrt{\langle S_1\cup S_2\rangle}]=k[\sqrt{\langle S_1\rangle}]\cup k[\sqrt{\langle S_2\rangle}]$.
\end{corollary}
\begin{proof}
Clearly $k[\sqrt{\langle S_1\cup S_2\rangle}]\supseteq k[\sqrt{\langle S_1\rangle}]\cup k[\sqrt{\langle S_2\rangle}]$. We only need to show that $k[\sqrt{\langle S_1\rangle}\cup\sqrt{\langle S_2\rangle}]$ is already a radical well-mixed $\D$-ideal.

Obviously, $\sqrt{\langle S_1\rangle}\cup\sqrt{\langle S_2\rangle}$ is a character set. Let $\bu,\bv\in\N[x]^n$ such that $\mathbf{u}+\mathbf{v}\in \sqrt{\langle S_1\rangle}\cup\sqrt{\langle S_2\rangle}$, then $\mathbf{u}+\mathbf{v}\in \sqrt{\langle S_1\rangle}$ or $\sqrt{\langle S_2\rangle}$ and hence $\mathbf{u}+x\mathbf{v}\in \sqrt{\langle S_1\rangle}$ or $\sqrt{\langle S_2\rangle}$. So $\mathbf{u}+x\mathbf{v}\in \sqrt{\langle S_1\rangle}\cup\sqrt{\langle S_2\rangle}$ which proves the condition (b) of Corollary \ref{drmi-cor2}. Similarly for the condition (a) of Corollary \ref{drmi-cor2}. Therefore, the corollary follows from Corollary \ref{drmi-cor2}.
\end{proof}

\begin{corollary}\label{drmi-cor3}
Let $I,J$ be two monomial $\D$-ideals. Then $\langle I+J\rangle_r=\langle I\rangle_r+\langle J\rangle_r$.
\end{corollary}
\begin{proof}
Suppose $I=k[S_1],J=k[S_2]$. Then $\langle I+J\rangle_r=k[\sqrt{\langle S_1\cup S_2\rangle}]$ and $\langle I\rangle_r+\langle J\rangle_r=k[\sqrt{\langle S_1\rangle}\cup\sqrt{\langle S_2\rangle}]$. So the equality follows from Corollary \ref{drmi-cor1}.
\end{proof}

\section{$\D$-Prime Decomposition of Perfect Monomial $\D$-Ideals}
It is well-known that in a $\D$-polynomial ring, any perfect $\D$-ideal is a finite intersection of $\D$-prime $\D$-ideals. In this section, we will give a $\D$-prime decomposition of perfect monomial $\D$-ideals in a $\D$-polynomial ring. The following lemma is taken from \cite[Proposition 1.2.20]{wibmer}.
\begin{proposition}\label{dpmi-prop}
Let $F$ and $G$ be subsets of any $\D$-ring $R$. Then
\[\{F\}\cap\{G\}=\{FG\}.\]
\end{proposition}
\begin{lemma}\label{dpmi-lemma}
Let $I$ be a perfect $\D$-ideal of $R$. Suppose that $\Y^{\mathbf{u}_1},\Y^{\mathbf{u}_2}$ are two monomials in $R$ such that $\Y^{\mathbf{u}_1+\mathbf{u}_2}\in I$. Then
\[I=\{I, \Y^{\mathbf{u}_1}\}\cap\{I, \Y^{\mathbf{u}_2}\}.\]
\end{lemma}
\begin{proof}
By Proposition \ref{dpmi-prop},
\[\{I, \Y^{\mathbf{u}_1}\}\cap\{ I, \Y^{\mathbf{u}_2}\}=\{ I, \Y^{\mathbf{u}_1+\mathbf{u}_2}\}=I.\]
\end{proof}

For $\mathbf{b}=(b_1,\ldots,b_n)\in \{0,1\}^n$, we define
\[\mathfrak{p}^{\mathbf{b}}:=[y_i\mid b_i\neq 0],\]
which is a $\D$-prime $\D$-ideal.
\begin{theorem}\label{dpmi-thm}
Let $I=\{\Y^{\bu}:\bu\in S\}$ where $S\subseteq \N[x]^n$. Then $I$ can be written as a finite intersection of $\D$-prime $\D$-ideals of the forms $\mathfrak{p}^{\mathbf{b}}$. That is, there exist $\mathbf{b}_1,\ldots,\mathbf{b}_s\in \{0,1\}^n$ such that
\[I=\mathfrak{p}^{\mathbf{b}_1}\cap\ldots\cap\mathfrak{p}^{\mathbf{b}_s}.\]
Moreover, if the decomposition is irredundant, then it is unique.
\end{theorem}
\begin{proof}
For $\bu=(u_1,\ldots,u_n)\in\N[x]^n$, define a vector $\mathbf{a}=(a_1,\ldots,a_n)\in \{0,1\}^n$ such that $a_i=1$ if $u_i\ne0$, or $a_i=0$, for $i=1,\ldots,n$. It is easy to see that if $\Y^{\bu}\in I$, then $\Y^{\mathbf{a}}\in I$. So without loss of generality, we can assume that $S\subseteq\{0,1\}^n$. By Lemma \ref{dpmi-lemma}, if a monomial $\Y^{\mathbf{u}}\in I$ and $\mathbf{u}=\mathbf{u}_1+\mathbf{u}_2$, then $I=\{I, \Y^{\mathbf{u}_1}\}\cap\{I, \Y^{\mathbf{u}_2}\}$. Iterating this process eventually write $I$ as follows:
\[I=\cap\{y_{i_1},\ldots, y_{i_k}\}=\cap[y_{i_1},\ldots, y_{i_k}].\]

After deleting unnecessary intersectands, we can assume that the intersection is irredundant. It is easy to see that this irredundant intersection is finite. Thus there exist $\mathbf{b}_1,\ldots,\mathbf{b}_s\in \{0,1\}^n$ such that
\[I=\mathfrak{p}^{\mathbf{b}_1}\cap\ldots\cap\mathfrak{p}^{\mathbf{b}_s}.\]

The uniqueness is similar to Theorem \ref{drdi-thm}.
\end{proof}

\begin{remark}
In fact, it is more straightforward to get the $\D$-prime decomposition of perfect monomial $\D$-ideals by using Theorem \ref{drdi-thm}. Assume that $S\subseteq\{0,1\}^n$. Then by Theorem \ref{drdi-thm}, $\langle \Y^{\bu}:\bu\in S\rangle_r=\cap\langle y_{i_1},\ldots, y_{i_k}\rangle_r=\cap[y_{i_1},\ldots, y_{i_k}]$. Since $[y_{i_1},\ldots, y_{i_k}]$ are $\D$-prime $\D$-ideals, it follows $\langle \Y^{\bu}:\bu\in S\rangle_r$ is a perfect $\D$-ideal. Thus
\[I=\{\Y^{\bu}:\bu\in S\}=\langle \Y^{\bu}:\bu\in S\rangle_r=\cap[y_{i_1},\ldots, y_{i_k}].\]
\end{remark}

\section{Alexander Duality of Monomial $\D$-Ideals}
\begin{definition}
Given two vectors $\mathbf{a},\mathbf{b}\in (\N\cup\{-1\})^n$ with $\mathbf{b}\leqslant\mathbf{a}$(that is, $b_i\leqslant a_i$ for $i=1,\ldots,n$), let $\mathbf{a}\backslash\mathbf{b}$ denote the vector whose $i^{th}$ coordinate is
\begin{equation*}
a_i\backslash b_i=\begin{cases}
a_i+1-b_i, & \textrm{if }b_i\geqslant 0;\\
-1, & \textrm{if }b_i=-1.
\end{cases}
\end{equation*}
Assume $I$ is a radical well-mixed monomial $\D$-ideal. If $\mathbf{a}$ is a vector in $(\N\cup\{-1\})^n$ satisfying $\mathbf{a}\geqslant\mathbf{b}$ for any character vector $\mathbf{b}$ of $I$, then the {\bf Alexander dual} of $I$ with respect to $\mathbf{a}$ is defined as
\[I^{[\mathbf{a}]}=\cap\{\mathfrak{m}^{\mathbf{a}\backslash\mathbf{b}}\mid \mathbf{b} \textrm{ is a character vector of } I\}.\]
\end{definition}

Note that for vectors $\mathbf{b}\leqslant\mathbf{a}$ in $(\N\cup\{-1\})^n$, $\mathbf{a}\backslash(\mathbf{a}\backslash\mathbf{b})=\mathbf{b}$.

As in Corollary \ref{drmi-cor2}, for $\mathbf{b}=(b_1,\ldots,b_n)\in (\N\cup\{-1\})^n$, let $x^{\mathbf{b}}=(x^{b_1},\ldots,x^{b_n})$ and set $x^{-1}=0$.
\begin{proposition}
Suppose that $I$ is a radical well-mixed monomial $\D$-ideal and $\mathbf{a}$ is a vector in $(\N\cup\{-1\})^n$ satisfying $\mathbf{a}\geqslant\mathbf{c}$ for any character vector $\mathbf{c}$ of $I$. If $\mathbf{b}\leqslant\mathbf{a}$, then $\Y^{x^{\mathbf{b}}}$ lies outside $I$ if and only if $\Y^{x^{\mathbf{a}-\mathbf{b}}}$ lies inside $I^{[\mathbf{a}]}$.
\end{proposition}
\begin{proof}
Assume that $\{\mathbf{c}_1,\ldots,\mathbf{c}_m\}$is the set of character vectors of $I$. Then $\Y^{x^{\mathbf{b}}}\notin I$ if and only if $\mathbf{b}\ngeqslant\mathbf{c}_i$, or equivalently, $\mathbf{a}-\mathbf{b}\nleqslant\mathbf{a}-\mathbf{c}_i$ for all $i$. This means that for each $i$, some coordinate of $\mathbf{a}-\mathbf{b}$ equals at least the corresponding coordinate of $\mathbf{a}+1-\mathbf{c}_i$. That is $\Y^{x^{\mathbf{a}-\mathbf{b}}}\in \mathfrak{m}^{\mathbf{a}+1-\mathbf{c}_i}$ for all $i$, i.e.\ $\Y^{x^{\mathbf{a}-\mathbf{b}}}\in \cap_{i=1}^m\mathfrak{m}^{\mathbf{a}+1-\mathbf{c}_i}$. Next, let us show that
\begin{equation}\label{drmi-equ}
\cap_{i=1}^m\mathfrak{m}^{\mathbf{a}+1-\mathbf{c}_i}=\cap_{i=1}^m\mathfrak{m}^{\mathbf{a}\backslash\mathbf{c}_i}+
\mathfrak{m}^{\mathbf{a}+2}=I^{[\mathbf{a}]}+\mathfrak{m}^{\mathbf{a}+2}.
\end{equation}

It is obvious that $\cap_{i=1}^m(\mathfrak{m}^{\mathbf{a}\backslash\mathbf{c}_i}+\mathfrak{m}^{\mathbf{a}+2})
\supseteq\cap_{i=1}^m\mathfrak{m}^{\mathbf{a}\backslash\mathbf{c}_i}+\mathfrak{m}^{\mathbf{a}+2}$. For the converse, choose $f\in
\cap_{i=1}^m(\mathfrak{m}^{\mathbf{a}\backslash\mathbf{c}_i}+\mathfrak{m}^{\mathbf{a}+2})$, then for each $i$, we could write $f=f_i+g_i$, where $f_i\in \mathfrak{m}^{\mathbf{a}\backslash\mathbf{c}_i}, g_i\in \mathfrak{m}^{\mathbf{a}+2}$. Thus $f^m=(f_1+g_1)\ldots(f_m+g_m)\in \cap_{i=1}^m\mathfrak{m}^{\mathbf{a}\backslash\mathbf{c}_i}+\mathfrak{m}^{\mathbf{a}+2}$.
By Corollary \ref{drmi-cor3}, $\cap_{i=1}^m\mathfrak{m}^{\mathbf{a}\backslash\mathbf{c}_i}+\mathfrak{m}^{\mathbf{a}+2}$
is radical, so $f\in\cap_{i=1}^m\mathfrak{m}^{\mathbf{a}\backslash\mathbf{c}_i}+\mathfrak{m}^{\mathbf{a}+2}$. Hence we have
$$\cap_{i=1}^m(\mathfrak{m}^{\mathbf{a}\backslash\mathbf{c}_i}+\mathfrak{m}^{\mathbf{a}+2})
=\cap_{i=1}^m\mathfrak{m}^{\mathbf{a}\backslash\mathbf{c}_i}+\mathfrak{m}^{\mathbf{a}+2}.$$
By Corollary \ref{drmi-cor3} again,
\begin{align*}
\mathfrak{m}^{\mathbf{a}\backslash\mathbf{c}_i}+\mathfrak{m}^{\mathbf{a}+2}&=\langle y_j^{x^{a_j+1-c_{ij}}}:c_{ij}\neq -1,j=1,\ldots,n\rangle_r+\langle y_j^{x^{a_j+2}}:j=1,\ldots,n\rangle_r\\
&=\langle y_j^{x^{a_j+1-c_{ij}}},y_j^{x^{a_j+2}}:c_{ij}\neq -1,j=1,\ldots,n\rangle_r\\
&=\langle y_j^{x^{a_j+1-c_{ij}}},j=1,\ldots,n\rangle_r=\mathfrak{m}^{\mathbf{a}+1-\mathbf{c}_i}.
\end{align*}
Thus (\ref{drmi-equ}) holds. Since $\mathbf{a}-\mathbf{b}\leqslant\mathbf{a}+1$, $\Y^{x^{\mathbf{a}-\mathbf{b}}}\in \cap_{i=1}^m\mathfrak{m}^{\mathbf{a}+1-\mathbf{c}_i}=
\cap_{i=1}^m\mathfrak{m}^{\mathbf{a}\backslash\mathbf{c}_i}+
\mathfrak{m}^{\mathbf{a}+2}=I^{[\mathbf{a}]}+\mathfrak{m}^{\mathbf{a}+2}$ exactly when $\Y^{x^{\mathbf{a}-\mathbf{b}}}\in
I^{[\mathbf{a}]}$.
\end{proof}

\begin{theorem}
Suppose that $I$ is a radical well-mixed monomial $\D$-ideal and $\mathbf{a}$ is a vector in $(\N\cup\{-1\})^n$ satisfying $\mathbf{a}\geqslant\mathbf{b}$ for any character vector $\mathbf{b}$ of $I$. Then $\mathbf{a}\geqslant\mathbf{c}$ for any character vector $\mathbf{c}$ of $I^{[\mathbf{a}]}$, and $(I^{[\mathbf{a}]})^{[\mathbf{a}]}=I$.
\end{theorem}
\begin{proof}
The proof is similar to the proof of Theorem 5.24 in p.90 in \cite{mi}.
\end{proof}

\begin{theorem}
Suppose that $I$ is a radical well-mixed monomial $\D$-ideal and $\mathbf{a}$ is a vector in $(\N\cup\{-1\})^n$ satisfying $\mathbf{a}\geqslant\mathbf{b}$ for any character vector $\mathbf{b}$ of $I$. Then
\[I=\cap\{\mathfrak{m}^{\mathbf{a}\backslash\mathbf{b}}\mid\mathbf{b} \textrm{ is a character vector of } I^{[\mathbf{a}]}\},\]
and
\[I^{[\mathbf{a}]}=\langle \Y^{x^{\mathbf{a}\backslash\mathbf{b}}}\mid \mathfrak{m}^{\mathbf{b}}\textrm{ is an irreducilbe component of } I\rangle_r.\]
\end{theorem}
\begin{proof}
The proof is similar to the proof of Theorem 5.27 in p.90 in \cite{mi}.
\end{proof}

\bibliographystyle{amsplain}

\end{document}